\documentclass[11pt]{amsart}
\usepackage[margin=1in]{geometry}
\usepackage{amssymb,amsfonts,amsmath,mathtools,amsthm}
\usepackage{enumerate}
\usepackage{mathrsfs}
\usepackage[unicode]{hyperref}
\usepackage[normalem]{ulem}
\let\C\undefined
\usepackage{constants}
\usepackage{bbm}
\usepackage[utf8]{inputenc}

\newtheorem{theorem}{Theorem}[section]

\newtheorem{proposition}[theorem]{Proposition}
\newtheorem{lemma}[theorem]{Lemma}
\theoremstyle{remark}

\newtheorem*{remarks}{Remarks}

\numberwithin{equation}{section}

\newconstantfamily{abcon}{symbol=c}

\title[Exceptional zeros of $\mathrm{GL}_3\times\mathrm{GL}_3$ Rankin--Selberg $L$-functions]{Exceptional zeros of $\mathrm{GL}_3\times\mathrm{GL}_3$ Rankin--Selberg $L$-functions}
\author{Jesse Thorner}
\address{Department of Mathematics, University of Illinois, Urbana, IL 61801, USA}
\email{\href{mailto:jesse.thorner@gmail.com}{jesse.thorner@gmail.com}}

\begin{document}

\begin{abstract}
Let $\chi$ be an idele class character over a number field $F$, and let $\pi,\pi'$ be any two cuspidal automorphic representations of $\mathrm{GL}_2(\mathbb{A}_F)$.  We prove that the Rankin--Selberg $L$-function $L(s,\mathrm{Sym}^2(\pi)\times(\mathrm{Sym}^2 (\pi')\otimes\chi))$ has a ``standard'' zero-free region with no exceptional Landau--Siegel zero except possibly when it is divisible by the $L$-function of a real idele class character.  In particular, no such zero exists if $\pi$ is non-dihedral and $\pi'$ is not a twist of $\pi$.  Until now, this was only known when $\pi=\pi'$, $\pi$ is self-dual, and $\chi$ is trivial.
\end{abstract}

\maketitle

\section{Introduction and statement of the main results}
\label{sec:intro}

Let $F$ be a number field, $\mathbb{A}_F$ be the ring of adeles over $F$, and $\mathfrak{F}_{n}$ be the set of cuspidal automorphic representations $\pi$ of $\mathrm{GL}_{n}(\mathbb{A}_F)$ with unitary central character $\omega_{\pi}$.  Let $C(\pi)$ be the analytic conductor of $\pi$ (see \eqref{eqn:AC_def}), $L(s,\pi)$ the standard $L$-function, and $\widetilde{\pi}\in\mathfrak{F}_{n}$ the contragredient of $\pi$.  The generalized Riemann hypothesis (GRH) predicts that if $\pi\in\mathfrak{F}_{n}$ and $\mathrm{Re}(s)>\frac{1}{2}$, then $L(s,\pi)\neq 0$.  Jacquet and Shalika \cite{JS} proved that if $\mathrm{Re}(s)\geq 1$, then $L(s,\pi)\neq 0$, extending the work of Hadamard and de la Vall{\'e}e Poussin for the Riemann zeta function.  Let $|\cdot|$ denote the idelic norm, and let $t\in\mathbb{R}$.  Replacing $\pi$ with $\pi\otimes|\cdot|^{it}$ and varying $t$, we find that it is equivalent to prove that if $\pi\in\mathfrak{F}_{n}$ and $\sigma\geq 1$, then $L(\sigma,\pi)\neq 0$.  By \cite[Theorem 1.1]{Wattanawanichkul},  there exists an absolute constant $\Cl[abcon]{ZFR1}>0$ such that if $\pi\neq\widetilde{\pi}$, then
\begin{equation}
\label{eqn:ZFR}
L(\sigma,\pi)\neq 0,\qquad \sigma\geq  1-\Cr{ZFR1}/(n\log C(\pi)).
\end{equation}
If $\pi=\widetilde{\pi}$, then there is at most one zero (necessarily simple) in the interval \eqref{eqn:ZFR}.  If $L(s,\pi)$ has a zero in the region \eqref{eqn:ZFR}, then $L(s,\pi)$ has an {\it exceptional zero}.

Given $(\pi,\pi')\in\mathfrak{F}_{n}\times\mathfrak{F}_{n'}$, let $L(s,\pi\times\pi')$ be the Rankin--Selberg $L$-function, as defined by Jacquet, Piatetski-Shapiro, and Shalika \cite{JPSS}.  Shahidi \cite{Shahidi} proved that if $\mathrm{Re}(s)\geq 1$, then $L(s,\pi\times\pi')\neq 0$.  Replacing $\pi$ with $\pi\otimes|\cdot|^{it}$ and varying $t\in\mathbb{R}$, we find that it suffices to prove that if $(\pi,\pi')\in\mathfrak{F}_{n}\times\mathfrak{F}_{n'}$ and $\sigma\geq 1$, then $L(\sigma,\pi\times\pi')\neq 0$.  The only fully uniform improvement is due to Brumley (\cite{Brumley} and \cite[Appendix]{Lapid})---there exists an effectively computable constants $\Cl[abcon]{Brumley}=\Cr{Brumley}(n,n',F,\varepsilon)>0$ such that
\[
L(\sigma,\pi\times\pi')\neq 0,\qquad \sigma\geq 1-\Cr{Brumley}/(C(\pi)C(\pi'))^{n+n'}.
\]
Harcos and the author \cite[Theorem 1.1]{HarcosThorner} proved that for all $\varepsilon>0$, there exists an ineffective constant $\Cl[abcon]{Siegel}=\Cr{Siegel}(\pi,\pi',\varepsilon)>0$ such that if $\chi\in\mathfrak{F}_{1}$, then
\[
L(\sigma,\pi\times(\pi'\otimes\chi))\neq 0,\qquad \sigma\geq 1-\Cr{Siegel}/C(\chi)^{\varepsilon}.
\]

Langlands conjectured that $L(s,\pi\times\pi')$ is modular, i.e., $L(s,\pi\times\pi')$ factors as a product of standard $L$-functions.  As of now, this is only known in special cases, most notably when $\pi\in\mathfrak{F}_{2}$ and $\pi'\in\mathfrak{F}_{2}\cup\mathfrak{F}_{3}$ \cite{KimShahidi,RamakrishnanWang}.  Hoffstein and Ramakrishnan \cite{HoffsteinRamakrishnan} proved that if all Rankin--Selberg $L$-functions are modular and $\pi \in \cup_{n=2}^{\infty}\mathfrak{F}_{n}$, then $L(s,\pi)$ has no zero in \eqref{eqn:ZFR}, eliminating exceptional zeros.  Modularity is known only in special cases, so unconditionally eliminating exceptional zeros remains a difficult and fruitful problem.  For a Rankin--Selberg $L$-function $L(s,\pi\times\pi')$, we say that $L(s,\pi\times\pi')$ has {\it no exceptional zero} if there exists an effectively computable constant $\Cl[abcon]{ZFR11}=\Cr{ZFR11}(n,n')>0$ such that
\[
L(\sigma,\pi\times\pi')\neq 0,\qquad \sigma \geq 1-\Cr{ZFR11}/\log(C(\pi)C(\pi')).
\]

Further discussion requires symmetric power lifts.  Let $v$ be a place of $F$, and let $F_v$ be the completion of $F$ relative to $v$.  Given $\pi\in\mathfrak{F}_{2}$, we express $\pi$ as a restricted tensor product $\bigotimes_v\pi_v$ of smooth, admissible representations of $\mathrm{GL}_2(F_v)$.  If $m\geq 0$, then $\mathrm{Sym}^m\colon \mathrm{GL}_2(\mathbb{C})\to\mathrm{GL}_{m+1}(\mathbb{C})$ is the $(m+1)$-dimensional irreducible representation of $\mathrm{GL}_2(\mathbb{C})$ on symmetric tensors of rank $m$.  If $P(x,y)$ is a homogeneous degree $m$ polynomial in two variables and $g\in\mathrm{GL}_2(\mathbb{C})$, then $\mathrm{Sym}^m(\pi')\in\mathrm{GL}_{m+1}(\mathbb{C})$ is the matrix giving the change in coefficients of $P$ under the change of variables by $g$.  Let $\varphi_v$ be the two-dimensional representation of the Deligne--Weil group attached to $\pi_v$ and $\mathrm{Sym}^m(\pi_v)$ be the smooth admissible representation of $\mathrm{GL}_{m+1}(F_v)$ attached to the representation $\mathrm{Sym}^m\circ \varphi_v$.  By the local Langlands correspondence, $\mathrm{Sym}^m(\pi_v)$ is well-defined for every place $v$ of $F$.

Langlands conjectured that $\mathrm{Sym}^m(\pi) = \bigotimes_v \mathrm{Sym}^m(\pi_v)$ is an automorphic representation of $\mathrm{GL}_{m+1}(\mathbb{A}_F)$, possibly not cuspidal.  This is known for $m\leq 4$ \cite{GJ,Kim,KimShahidi}.

Given $\pi,\pi'\in\mathfrak{F}_{n}$, we write $\pi\sim\pi'$ if there exists $\psi\in\mathfrak{F}_{1}$ such that $\pi'=\pi\otimes\psi$.  Otherwise, we write $\pi\not\sim\pi'$.  Let $\mathbbm{1}\in\mathfrak{F}_{1}$ be the trivial representation, whose $L$-function is the Dedekind zeta function $\zeta_F(s)$.  If $\chi\in\mathfrak{F}_1$, then we call $L(s,\chi)$ an {\it abelian $L$-function}.
\begin{proposition}
\label{prop:list}
The following $L$-functions have no exceptional zero:
\begin{enumerate}[(i)]
	\item \cite{HoffsteinRamakrishnan,Wattanawanichkul} $L(s,\pi)$, where $\pi\in\cup_{n=1}^{\infty}\mathfrak{F}_{n}$ is not self-dual;
	\item \cite{Banks,HoffsteinRamakrishnan} $L(s,\pi)$, where $\pi\in\mathfrak{F}_{2}\cup \mathfrak{F}_{3}$;
	\item \cite{Luo} $L(s,\pi)$, where $\pi\in\mathfrak{F}_n$ and there exists $\psi\in\mathfrak{F}_1$ such that $\pi\otimes\psi=\pi$;
	\item \cite{RamakrishnanWang} $L(s,\pi\times\pi')$, where $\pi,\pi'\in\mathfrak{F}_{2}$ and no self-dual abelian $L$-functions divide it (which is necessarily the case when $\pi,\pi'$ are non-dihedral and $\pi\not\sim\pi'$);
	\item \cite{RamakrishnanWang} $L(s,\mathrm{Sym}^2(\pi)\times\mathrm{Sym}^2(\pi))/\zeta_F(s)=L(s,\mathrm{Sym}^2(\pi)\otimes\omega_{\pi})L(s,\mathrm{Sym}^4(\pi))$, where $\pi\in\mathfrak{F}_{2}$ is self-dual (in which case $\omega_{\pi}^2=\mathbbm{1}$) and no self-dual abelian $L$-functions divide it (which is necessarily the case when $\pi$ is not dihedral, tetrahedral, or octahedral);
	\item \cite{Luo} $L(s,\pi\times\pi')$, where $(\pi,\pi')\in\mathfrak{F}_{2}\times\mathfrak{F}_{3}$ and $\pi'\not\sim\mathrm{Sym}^2(\pi)$ when $\pi$ is non-dihedral;
	\item \cite{HumphriesThorner,Wattanawanichkul} $L(s,\pi\times(\widetilde{\pi}\otimes|\cdot|^{it}))$, where $\pi\in\mathfrak{F}_n$ and $t\neq 0$;
	\item \cite{Wattanawanichkul} $L(s,\pi\times\pi')$, where $(\pi,\pi')\in\mathfrak{F}_{n} \times\mathfrak{F}_{n'}$ and one of $\pi$ and $\pi'$ is self-dual.
\end{enumerate}
\end{proposition}

Our main result adds to this list.  In the main case, our result applies to Rankin--Selberg $L$-functions that are not yet known to be modular.

\begin{theorem}
\label{thm}
Let $(\pi,\pi',\chi)\in\mathfrak{F}_{2}\times\mathfrak{F}_{2}\times\mathfrak{F}_{1}$.  There exists an absolute and effectively computable constant $\Cl[abcon]{ZFR_thm}>0$ such that $L(\sigma,\mathrm{Sym}^2(\pi)\times(\mathrm{Sym}^2(\pi')\otimes\chi))\neq 0$ when $\sigma\geq 1-\Cr{ZFR_thm}/\log(C(\pi)C(\pi')C(\chi))$, provided that no self-dual abelian $L$-function divides it.
\end{theorem}

\begin{remarks}
\begin{enumerate}
\item If $F$ is totally real, $\chi\in\mathfrak{F}_1$, $m,n\geq 0$ with $\max\{m,n\}\geq 1$, and $\pi,\pi'\in\mathfrak{F}_2$ correspond with primitive Hilbert cusp forms, then $\mathrm{Sym}^m(\pi)\in\mathfrak{F}_{m+1}$ and $\mathrm{Sym}^n(\pi')\in\mathfrak{F}_{n+1}$ by work of Newton and Thorne \cite{NewtonThorne3}.  In this setting, the author \cite{Thorner_Siegel} has established the nonexistence of exceptional zeros for $L(s,\mathrm{Sym}^m(\pi)\times(\mathrm{Sym}^n(\pi')\otimes\chi))$.  The proof therein relies crucially on the the hypotheses that $F$ is totally real and $\pi$ and $\pi'$ are regular algebraic, whereas Theorem \ref{thm} holds without any such hypothesis.  In particular, Theorem \ref{thm} holds when $\pi$ and $\pi'$ correspond with Hecke--Maa{\ss} newforms.
	\item If $\pi,\pi'\in\mathfrak{F}_2$, $\pi'=\pi=\widetilde{\pi}$, and $\chi=\mathbbm{1}$, then Theorem \ref{thm} recovers Proposition \ref{prop:list}(v).
	\item Let $t\in\mathbb{R}$.  Replacing $\chi$ with $\chi|\cdot|^{it}$, we conclude that there exists an absolute and effectively computable constant $\Cl[abcon]{ZFR_thm2}>0$ such that if $L(s,\mathrm{Sym}^2(\pi)\times(\mathrm{Sym}^2(\pi')\otimes\chi))$ has no self-dual abelian $L$-functions in its factorization, then
\[
L(\sigma+it,\mathrm{Sym}^2(\pi)\times(\mathrm{Sym}^2(\pi')\otimes\chi))\neq 0,\quad \sigma\geq 1-\Cr{ZFR_thm2}/\log(C(\pi)C(\pi')C(\chi)(|t|+3)^{[F:\mathbb{Q}]}).
\]
\item Our proof handles 11 cases, classifying the possible factors arising from self-dual abelian $L$-functions.  In particular, if $(\pi,\pi',\chi)\in\mathfrak{F}_{2}\times\mathfrak{F}_{2}\times\mathfrak{F}_{1}$, $\pi$ is non-dihedral, and $\pi\not\sim\pi'$, then there are no self-dual abelian $L$-functions in the factorization.
\item The approach to zero-free regions and exceptional zeros that underlies each part of Proposition \ref{prop:list}	fails to work in the most difficult case of Theorem \ref{thm}, where $\pi$ and $\pi'$ are not of solvable polyhedral type and $\pi\not\sim\pi'$.  Given the technical nature of this failure, we discuss this later, in Section \ref{sec:strategy}.
\end{enumerate}
\end{remarks}

\subsection*{Acknowledgements}

The author thanks Jeffrey Hoffstein for helpful conversations.  The author is partially funded by the Simons Foundation (MP-TSM-00002484) and the National Science Foundation (DMS-2401311).

\section{Analytic properties of $L$-functions}
\label{sec:Properties}

Let $F$ be a number field with absolute discriminant $D_F$.  For a place $v$ of $F$, let $F_v$ be the associated completion.  For each place $v$ of $F$, we write $v\mid\infty$ (resp. $v\nmid \infty$) if $v$ is archimedean (resp. non-archimedean).  For $v\nmid\infty$, $q_v$ is the cardinality of the residue field of the local ring of integers $\mathcal{O}_v\subseteq F_v$, and $\varpi_v$ is the uniformizer.  The properties of $L$-functions given here rely on \cite{GodementJacquet,JPSS,MoeglinWaldspurger}.  In our use of $f=O(g)$ or $f\ll g$, the implied constant is always absolute and effectively computable.

\subsection{Standard $L$-functions}
\label{subsec:standard}

Let $\pi\in\mathfrak{F}_{n}$, let $\widetilde{\pi}\in\mathfrak{F}_{n}$ be the contragredient, and let $\omega_{\pi}$ the central character.  We express $\pi$ as a restricted tensor product $\bigotimes_v \pi_v$ of smooth admissible representations of $\mathrm{GL}_n(F_v)$.  Let $\delta_{\pi}=1$ if $\pi=\mathbbm{1}$ and $\delta_{\pi}=0$ otherwise.  Define the sets $S_{\pi}=\{v\nmid\infty\colon \textup{$\pi_v$ ramified}\}$ and $S_{\pi}^{\infty}=S_{\pi}\cup\{v\mid\infty\}$.  Let $N_{\pi}$ be the norm of the conductor of $\pi$.  If $v\nmid\infty$, then there are $n$ Satake parameters $(\alpha_{j,\pi}(v))_{j=1}^n$ such that
	\[
	L(s,\pi)=\prod_{v\nmid\infty}L(s,\pi_{v}),\qquad L(s,\pi_{v}) = \prod_{j=1}^{n}\frac{1}{1-\alpha_{j,\pi}(v)q_v^{-s}}
	\]
	converges absolutely for $\mathrm{Re}(s)>1$.  If $v\in S_{\pi}$, then at least one of the $\alpha_{j,\pi}(p)$ equals zero.
	
If $v\mid\infty$, then $(\mu_{j,\pi}(v))_{j=1}^n$ are the Langlands parameters at $v$, from which we define
\[
\Gamma_v(s)=\begin{cases}
	\pi^{-s/2}\Gamma(s/2)&\mbox{if $F_v=\mathbb{R}$,}\\
	2(2\pi)^{-s}\Gamma(s)&\mbox{if $F_v=\mathbb{C}$.}
	\end{cases}
	\]
and
\[
L(s,\pi_{\infty}) = \prod_{v\mid\infty}L(s,\pi_v) = \prod_{v\mid\infty}\prod_{j=1}^n \Gamma_v(s+\mu_{j,\pi}(v))
	\]
	The completed $L$-function $\Lambda(s,\pi)=(s(1-s))^{\delta_{\pi}}(D_F^n N_{\pi})^{\frac{s}{2}}L(s,\pi)L(s,\pi_{\infty})$ is entire of order $1$.  Since $\{\alpha_{j,\widetilde{\pi}}(v)\}=\{\overline{\alpha_{j,\pi}(v)}\}$, $N_{\widetilde{\pi}}=N_{\pi}$, and $\{\mu_{j,\widetilde{\pi}}(v)\}=\{\overline{\mu_{j,\pi}(v)}\}$, we have that $L(s,\widetilde{\pi})=\overline{L(\overline{s},\pi)}$.  The analytic conductor is
\begin{equation}
\label{eqn:AC_def}
C(\pi)=D_F^n N_{\pi} \prod_{v\mid\infty}\prod_{j=1}^n (|\mu_{j,\pi}(v)|+3)^{[F_v:\mathbb{R}]}.
\end{equation}
By \cite{LRS,MullerSpeh} there exists $\theta_n\in[0,\frac{1}{2}-\frac{1}{n^2+1}]$ such that
\begin{equation}
\label{eqn:Ramanujan1}
|\alpha_{j,\pi}(v)|\leq q_v^{\theta_{n}},\qquad \mathrm{Re}(\mu_{j,\pi}(v))\geq -\theta_{n}.
\end{equation}
We define $a_{\pi}(v^{\ell})$ by the Dirichlet series identity
\[
-\frac{L'}{L}(s,\pi)=\sum_{v}\sum_{\ell=1}^{\infty}\frac{\sum_{j=1}^n \alpha_{j,\pi}(v)^{\ell}\log q_v}{q_v^{\ell s}}=\sum_{v\nmid\infty}\sum_{\ell=1}^{\infty}\frac{a_{\pi}(v^{\ell})\log q_v}{q_v^{\ell s}}.
\]

\subsection{Rankin--Selberg $L$-functions}
\label{subsec:RS}

Let $\pi\in\mathfrak{F}_{n}$ and $\pi'\in\mathfrak{F}_{n'}$.  Let
\[
\delta_{\pi\times\pi'}=\begin{cases}
1&\mbox{if $\pi'=\widetilde{\pi}$,}\\
0&\mbox{otherwise.}
\end{cases}
\]
For each $v\notin S_{\pi}^{\infty}\cup S_{\pi'}^{\infty}$, define
\[
L(s,\pi_{v}\times\pi_{v}')=\prod_{j=1}^n \prod_{j'=1}^{n'}\frac{1}{1-\alpha_{j,\pi}(v)\alpha_{j',\pi'}(v)q_v^{-s}}.
\]
Jaquet, Piatetski-Shapiro, and Shalika proved the following theorem.

\begin{theorem}\cite{JPSS}
\label{thm:JPSS}
If $(\pi,\pi')\in\mathfrak{F}_{n}\times\mathfrak{F}_{n'}$, then there exist
\begin{enumerate}
\item complex numbers $(\alpha_{j,j',\pi\times\pi'}(v))_{j=1}^n{}_{j'=1}^{n'}$ for each $v\in S_{\pi}\cup S_{\pi'}$, from which we define
	\begin{align*}
	L(s,\pi_v\times\pi_v') = \prod_{j=1}^n \prod_{j'=1}^{n'}\frac{1}{1-\alpha_{j,j',\pi\times\pi'}(v)q_v^{-s}},\quad L(s,\widetilde{\pi}_v\times\widetilde{\pi}_v') = \prod_{j=1}^n \prod_{j'=1}^{n'}\frac{1}{1-\overline{\alpha_{j,j',\pi\times\pi'}(v)}q_v^{-s}};
	\end{align*}
\item complex numbers $(\mu_{j,j',\pi\times\pi'}(v))_{j=1}^{n}{}_{j'=1}^{n'}$ for each $v\mid\infty$, from which we define
	\begin{align*}
	L(s,\pi_{v}\times\pi_{v}') = \prod_{j=1}^n \prod_{j'=1}^{n'}\Gamma_v(s+\mu_{j,j',\pi\times\pi'}(v)),\quad L(s,\widetilde{\pi}_{v}\times\widetilde{\pi}_{v}') = \prod_{j=1}^n \prod_{j'=1}^{n'}\Gamma_v(s+\overline{\mu_{j,j',\pi\times\pi'}(v)});
	\end{align*}
\item a conductor, an integral ideal whose norm is denoted $N_{\pi\times\pi'}=N_{\widetilde{\pi}\times\widetilde{\pi}'}$; and
\item a complex number $W(\pi\times\pi')$ of modulus $1$
\end{enumerate}
such that the Rankin--Selberg $L$-functions
\[
L(s,\pi\times\pi')=\prod_{v\nmid\infty}L(s,\pi_v\times\pi_v'),\qquad L(s,\widetilde{\pi}\times\widetilde{\pi}')=\prod_{v\nmid\infty}L(s,\widetilde{\pi}_v\times\widetilde{\pi}_v')
\]
converge absolutely for $\mathrm{Re}(s)>1$, the completed $L$-functions
\begin{align*}
\Lambda(s,\pi\times\pi') &= (s(1-s))^{\delta_{\pi\times\pi'}} (D_F^{n'n}N_{\pi\times\pi'})^{\frac{s}{2}} L(s,\pi\times\pi')\prod_{v\mid\infty}L(s,\pi_v\times\pi_{v}')\\
\Lambda(s,\widetilde{\pi}\times\widetilde{\pi}') &= (s(1-s))^{\delta_{\pi\times\pi'}} (D_F^{n'n}N_{\pi\times\pi'})^{\frac{s}{2}}L(s,\widetilde{\pi}\times\widetilde{\pi}')\prod_{v\mid\infty}L(s,\widetilde{\pi}_{v}\times\widetilde{\pi}_{v}')
\end{align*}
are entire of order $1$, and $\Lambda(s,\pi\times\pi')=W(\pi\times\pi')\Lambda(1-s,\widetilde{\pi}\times\widetilde{\pi}')$.
\end{theorem}

It follows from Theorem \ref{thm:JPSS} that
\begin{equation}
	\label{eqn:dual}
	L(s,\widetilde{\pi}\times\widetilde{\pi}')=\overline{L(\overline{s},\pi\times\pi)}.
	\end{equation}
The following bounds hold:
	\begin{equation}
		\label{eqn:Ramanujan2}
		\begin{aligned}
	|\alpha_{j,j',\pi\times\pi'}(v)|\leq  q_v^{\theta_n+\theta_{n'}},\qquad \mathrm{Re}(\mu_{j,j',\pi\times\pi'}(v))\geq -(\theta_n+\theta_{n'}).
	\end{aligned}
	\end{equation}
	If $\ell\geq 1$ is an integer and $v\nmid\infty$, then we define
	\begin{equation}
	\label{eqn:a_def}
	\begin{aligned}
	a_{\pi\times\pi'}(v^{\ell})&= \begin{cases}
 	a_{\pi}(v^{\ell})a_{\pi'}(v^{\ell})&\mbox{if $v\notin S_{\pi}\cup S_{\pi'}$,}\\
 	\sum_{j=1}^n \sum_{j'=1}^{n'}\alpha_{j,j',\pi\times\pi'}(v)^{\ell}&\mbox{if $v\in S_{\pi}\cup S_{\pi'}$,}
 \end{cases}\\
 a_{\widetilde{\pi}\times\widetilde{\pi}'}(v^{\ell})&=\overline{a_{\pi\times\pi'}(v^{\ell})}.
 \end{aligned}
	\end{equation}
	We have the Dirichlet series identity
	\begin{equation}
		\label{eqn:log_deriv}
	-\frac{L'}{L}(s,\pi\times\pi')=\sum_{v\nmid\infty}\sum_{\ell=1}^{\infty}\frac{a_{\pi\times\pi'}(v^{\ell})\log q_v}{q_v^{\ell s}},\qquad\mathrm{Re}(s)>1.
	\end{equation}
	
\subsection{Isobaric sums}
\label{subsec:isobaric}

Let $r\geq 1$ be an integer.  For $1\leq j\leq r$, let $\pi_{j}\in\mathfrak{F}_{d_j}$.  Langlands associated to $(\pi_1,\ldots,\pi_r)$ an automorphic representation of $\mathrm{GL}_{d_1+\cdots+d_r}(\mathbb{A}_F)$, the isobaric sum $\Pi=\pi_1\boxplus\cdots\boxplus\pi_r$.  Its $L$-function is $L(s,\Pi)=\prod_{j=1}^r L(s,\pi_j)$, and its contragredient is $\widetilde{\pi}_1\boxplus\cdots\boxplus\widetilde{\pi}_r$.  Let $\mathfrak{A}_{n}$ be the set of isobaric automorphic representations of $\mathrm{GL}_n(\mathbb{A}_F)$.  If $\Pi=\pi_1\boxplus\cdots\boxplus\pi_r\in\mathfrak{A}_n$ and $\Pi'=\pi_1'\boxplus\cdots\boxplus\pi_{r'}'\in\mathfrak{A}_{n'}$, then
\[
L(s,\Pi\times\Pi')=\prod_{j=1}^r \prod_{k=1}^{r'} L(s,\pi_j\times\pi_k').
\]
\begin{lemma}\cite{HoffsteinRamakrishnan}
\label{lem:HR}
If $\Pi\in\mathfrak{A}_{n}$, then $-\frac{L'}{L}(s,\Pi\times\widetilde{\Pi})$ has non-negative Dirichlet coefficients.
\end{lemma}

\subsection{Real zeros}
	
We define the analytic conductor
	\[
	C(\pi\times\pi')=D_F^{n'n}N_{\pi\times{\pi}'}\prod_{v\mid\infty}\prod_{j=1}^n \prod_{j=1}^{n'} (|\mu_{j,j',\pi\times{\pi}'}(v)|+3)^{[F_v:\mathbb{Q}]}.
	\]
	Using \cite[Theorem 2]{BushnellHenniart} and \cite[Lemma A.1]{Wattanawanichkul}, we infer that
	\begin{equation}
		\label{eqn:BH}
		N_{\pi\times\pi'}\mid N_{\pi}^{n'}N_{\pi'}^{n},\qquad \C(\pi\times(\pi'\otimes\chi))\leq C(\pi)^{n'}C(\pi')^{n}C(\chi)^{n'n}.
	\end{equation}

\begin{lemma}
	\label{lem:GHL}
Let $J\geq 1$.  For $j\in\{1,\ldots,J\}$, let $(\pi_j,\pi_j',\chi_j)\in\mathfrak{F}_{n_j}\times\mathfrak{F}_{n_j'}\times\mathfrak{F}_{1}$.  Define
\[
\mathfrak{Q} = \prod_{j = 1}^
   J C (\pi_j) C (\pi_j') C (\chi),\quad \mathfrak {S} = \bigcup_ {j = 1}^J (S_{\pi_j}\cup S_{\pi_j'}\cup S_{\chi}),\quad D(s) =  \prod_{j = 1}^J L(s, \pi_j\times(\pi' _j\otimes\chi_j)).
\]
Assume that $D(s)$ is holomorphic on $\mathbb{C}-\{1\}$ with a pole of order $r\geq 1$ at $s=1$.  Write
\begin{equation}
\label{eqn:aDdef}
a_D(v^{\ell})=\sum_{j=1}^J a_{\pi_j\times(\pi_j'\otimes\chi)}(v^{\ell}),\qquad -\frac{D'}{D}(s) = \sum_{v\nmid\infty}\sum_{\ell=1}^{\infty}\frac{a_D(v^{\ell})\log q_v}{q_v^{\ell s}}.
\end{equation}
Let $Q \geq \mathfrak{Q}$.  There exists an effectively computable constant $\Cl[abcon]{GHL}>0$ (depending only on the numbers $n_j$, $n_j'$, and $r$) such that if $\mathrm{Re}(a_D(v^{\ell}))\geq 0$ for all $\ell\geq 1$ and $v\notin\mathfrak{S}$, then $D(\sigma)$ has no zeros in the interval $[1,\infty)$ and at most $r$ zeros in the interval $[1-\Cr{GHL}/\log Q,1)$.
\end{lemma}
\begin{proof}
The proof is identical to \cite[Lemma 5.9]{IK}.
\end{proof}

\subsection{Symmetric power lifts from $\mathrm{GL}_2$}

Let $\pi\in\mathfrak{F}_{2}$.  Define $\mathrm{Sym}^0(\pi)=\mathbbm{1}$ and $\mathrm{Sym}^1(\pi)=\pi$.  If is conjectured that if $m\geq 2$, then $\mathrm{Sym}^{m}(\pi)\in\mathfrak{A}_{m+1}$ with contragredient
\[
\mathrm{Sym}^{m}(\widetilde{\pi})=\mathrm{Sym}^{m}(\pi)\otimes\overline{\omega}_{\pi}^{m}.
\]
This is known for $m\leq 4$ \cite{GJ,KimShahidi,Kim}.  If $\mathrm{Sym}^{m}(\pi)\in\mathfrak{A}_{m+1}$, then for each $v\in S_{\pi}$, there exist complex numbers $(\alpha_{j,\mathrm{Sym}^{m}(\pi)}(v))_{j=0}^{m}$ such that
\begin{equation}
\label{eqn:symm_power_def}
L(s,\mathrm{Sym}^m(\pi_v))=\begin{cases}
\displaystyle\prod_{j=0}^m \frac{1}{1-\alpha_{1,\pi}(v)^{j}\alpha_{2,\pi}(v)^{m-j}q_v^{-s}}&\mbox{if $v\nmid\infty$ and $v\notin S_{\pi}$,}\vspace{2mm}\\
\displaystyle\prod_{j=0}^m\frac{1}{1-\alpha_{j,\mathrm{Sym}^{m}(\pi)}(v)q_v^{-s}}&\mbox{if $v\nmid\infty$ and $v\in S_{\pi}$.}
\end{cases}	
\end{equation}

\begin{lemma}
\label{lem:CG}
Let $j,k,u\geq 0$ be integers, $\pi\in\mathfrak{F}_{2}$, and $\psi_1,\psi_2,\chi\in\mathfrak{F}_{1}$.  Suppose that if
\[
u\in\{j,k\}\cup\{j+k-2r\colon 0\leq r\leq \min\{j,k\}\},
\]
then $\mathrm{Sym}^u(\pi)\in\mathfrak{F}_{u+1}$.  Then
\[
\mathrm{Sym}^{j}(\pi\otimes\psi_1)\boxtimes\mathrm{Sym}^k(\pi\otimes\psi_2)\otimes\chi=\mathop{\boxplus}_{r=0}^{\min\{j,k\}}\mathrm{Sym}^{j+k-2r}(\pi)\otimes\chi\psi_1^j \psi_2^k\omega_{\pi}^{r}\in\mathfrak{A}_{(j+1)(k+1)},
\]
and $L(s,\mathrm{Sym}^{j}(\pi\otimes\psi_1)\boxtimes(\mathrm{Sym}^k(\pi\otimes\psi_2)\otimes\chi))=L(s,\mathrm{Sym}^{j}(\pi\otimes\psi_2)\times(\mathrm{Sym}^k(\pi\otimes\psi_2)\otimes\chi))$. In particular, if $v\nmid\infty$, then the following identities hold:
\begin{align}
a_{\mathrm{Sym}^{j}(\pi)\times(\mathrm{Sym}^k(\pi)\otimes\chi)}(v^{\ell}) &= \sum_{r=0}^{\min\{j,k\}}a_{\mathrm{Sym}^{j+k-2r}(\pi)\otimes\chi \omega_{\pi}^{r}}(v^{\ell}),\label{eqn:CG1}\\
a_{\mathrm{Sym}^{j}(\pi)\times\mathrm{Sym}^j(\widetilde{\pi})}(v^{\ell})&=1+\sum_{r=1}^j a_{\mathrm{Sym}^{2r}(\pi)\otimes\omega_{\pi}^{-r}}(v^{\ell}).\label{eqn:CG2}
\end{align}
\end{lemma}
\begin{proof}
	These follow from the Clebsch--Gordan identities.
\end{proof}

\subsection{The symmetric square lift from $\mathrm{GL}_n$}

Let $(\pi,\chi)\in\mathfrak{F}_{n}\times\mathfrak{F}_{1}$.  Let $S$ be a set of places containing $S_{\pi}^{\infty}\cup S_{\chi}^{\infty}$ and all places dividing $2$.  The $\chi$-twist of the partial $L$-function of the symmetric square representation $\mathrm{Sym}^2\colon \mathrm{GL}_n(\mathbb{C})\to\mathrm{GL}_{n(n+1)/2}(\mathbb{C})$ is
\begin{equation}
\label{eqn:SYM2def}
L^S(s,\pi;\mathrm{Sym}^2\otimes\chi)=\prod_{v\notin S}~\prod_{1\leq j\leq k\leq n}\frac{1}{1-\chi_v(\varpi_v)\alpha_{j,\pi}(v)\alpha_{k,\pi}(v)q_v^{-s}}.
\end{equation}
\begin{theorem}
\label{thm:sym2}
Let $(\pi,\chi)\in\mathfrak{F}_{n}\times\mathfrak{F}_{1}$.  If $S$ is a set of places containing $S_{\pi}^{\infty}\cup S_{\chi}^{\infty}$ and all places dividing $2$, then $L^S(s,\pi;\mathrm{Sym}^2\otimes\chi)$ is holomorphic on $\mathbb{C}$ except possibly for simple poles at $s\in\{0,1\}$.  If $\chi^n \omega_{\pi}^2\neq\mathbbm{1}$, then there is no pole.
\end{theorem}
\begin{proof}
This is \cite[Theorem 7.1]{Takeda}.
\end{proof}

\section{Preliminaries for the proof of Theorem \ref{thm}}
\label{sec:strategy}

Let $\chi\in\mathfrak{F}_1$.  Given $\pi\in\mathfrak{F}_2$, let $\mathrm{Ad}(\pi)=\mathrm{Sym}^2(\pi)\otimes\overline{\omega}_{\pi}$, the self-dual automorphic representation of $\mathrm{GL}_3(\mathbb{A}_F)$ whose $L$-function is $L(s,\pi\times\tilde{\pi})/\zeta_F(s)$.  Our main result is as follows.
\begin{theorem}
\label{thm:Ad}
Let $(\pi,\pi',\chi)\in\mathfrak{F}_{2}\times\mathfrak{F}_{2}\times\mathfrak{F}_{1}$.  There exists an absolute and effectively computable constant $\Cl[abcon]{ZFR_thm}>0$ such that $L(\sigma,\mathrm{Ad}(\pi)\times(\mathrm{Ad}(\pi')\otimes\chi))\neq 0$ when $\sigma\geq 1-\Cr{ZFR_thm}/\log(C(\pi)C(\pi')C(\chi))$,  provided that no self-dual abelian $L$-function divides it.
\end{theorem}
\begin{proof}[Proof of Theorem \ref{thm} assuming Theorem \ref{thm:Ad}]
This is immediate in light of the identity
\[
L(s,\mathrm{Sym}^2(\pi)\times(\mathrm{Sym}^2(\pi')\otimes\chi))=L(s,\mathrm{Ad}(\pi)\times(\mathrm{Ad}(\pi')\otimes\overline{\omega}_{\pi}\overline{\omega}_{\pi'}\chi)).\qedhere
\]	
\end{proof}

Let $(\pi_1,\pi_2)\in\mathfrak{F}_{n_1}\times\mathfrak{F}_{n_2}$.  Assume that $L(s,\pi_1\times\pi_2)$ is an irreducible $L$-function that is not the $L$-function of a real-valued idele class character.  The only known strategy to prove that $L(\sigma,\pi_1\times\pi_2)$ has no exceptional zero is to construct a Dirichlet series $D(s)$, integers $\ell_1,\ell_2\geq 0$ and $k\geq 1$ satisfying $\ell_1+\ell_2>k$, and a fixed $t\in(0,1)$ such that
\begin{enumerate}
	\item at each unramified $v\nmid\infty$ and $\ell\geq 1$, we have $\mathrm{Re}(a_D(v^{\ell}))\geq 0$ (see \eqref{eqn:aDdef}),
	\item $D(s)$ is holomorphic everywhere except for a pole of order $k$ at $s=1$, and
	\item $(s-1)^{k}D(s) L(s,\pi_1\times\pi_2)^{-\ell_1}L(s,\widetilde{\pi}_1\times\widetilde{\pi}_2)^{-\ell_2}$ is holomorphic at each real $s\in(t,1]$.
\end{enumerate}
A real zero of $L(s,\pi_1\times\pi_2)$ is a zero of $D(s)$ with multiplicity at least $\ell_1+\ell_2>k$, so the existence of a exceptional zero of $L(s,\pi_1\times\pi_2)$ contradicts Lemma \ref{lem:GHL} applied to $D(s)$.

\subsection{Earlier work}
\label{subsec:earlier_work}

First, we consider the case where $\pi_1\in\cup_{n=2}^{\infty}\mathfrak{F}_{n}$ and $\pi_2=\mathbbm{1}$.  Hoffstein and Ramakrishnan \cite[Proof of Theorem B]{HoffsteinRamakrishnan} showed that sufficient progress towards the modularity of Rankin--Selberg $L$-functions suffices to prove that $L(s,\pi_1)$ has no exceptional zero.  Assume that $L(s,\pi_1\times\widetilde{\pi}_1)$ is modular so that there exists an isobaric automorphic representation $\pi_1\boxtimes\widetilde{\pi}_1$ such that $L(s,\pi_1\times\widetilde{\pi}_1)=L(s,\pi_1\boxtimes\widetilde{\pi}_1)$.  By \cite[Lemma 4.4]{HoffsteinRamakrishnan}, $\pi_1\boxtimes\tilde{\pi}_1$ has a cuspidal constituent $\tau\notin\{\mathbbm{1},\pi_1,\widetilde{\pi}_1\}$.  If $L(s,\pi_1\times\tau)$ is modular, then by \cite[Proofs of Lemma 4.4 and Claim 4.5]{HoffsteinLockhart}, $L(s,\pi_1\times\tau)/L(s,\pi_1)$ is entire.  Therefore, subject to the modularity of $L(s,\pi_1\times\widetilde{\pi}_1)$ and $L(s,\pi_1\times\tau)$, Hoffstein and Ramakrishnan prove that if $\Pi=\mathbbm{1}\boxplus\widetilde{\tau}\boxplus\pi_1$, then $D(s)=L(s,\Pi\times\widetilde{\Pi})$ satisfies the above criteria (1)--(3).

Proposition \ref{prop:list}(ii,iv,v,vi) follow by proving the existence of $\pi_1\boxtimes\widetilde{\pi}_1$ and the existence of a cuspidal constituent $\tau$ of $\pi_1\boxtimes\widetilde{\pi}_1$ such that there exists an effectively computable constant $\Cl[abcon]{HR}=\Cr{HR}(n)\in(0,1)$ such that $L(\sigma,\pi_1\times\tau)/L(\sigma,\pi_1)$ is holomorphic at each $\sigma\in(\Cr{HR},1)$.  The holomorphy condition is most complicated in the proof of part (v).  In this case, $\pi\in \mathfrak{F}_{2}$ is self-dual and not dihedral or tetrahedral or octahedral, $\pi_1=\mathrm{Sym}^4(\pi)$, $\tau=\mathrm{Sym}^2(\pi)$, and
\[
L(\sigma,\pi_1\times\tau)/L(\sigma,\pi_1)=L(\sigma,\mathrm{Sym}^3(\pi);\mathrm{Sym}^2).
\]
Bump and Ginzburg \cite{BumpGinzburg} proved that this is holomorphic at $\sigma\in(1/2,1)$. Proposition \ref{prop:list}(i,iii,vii,viii) are proved by constructing a $D(s)$ satisfying (1)--(3) above using hypotheses related to self-duality or twist-equivalence.

For simplicity, let $\pi,\pi'\in\mathfrak{F}_2$ be non-dihedral, non-tetrahedral, and non-octahedral, with $\pi\not\sim\pi'$.  Let $\pi_1=\mathrm{Ad}(\pi)$ and $\pi_2=\mathrm{Ad}(\pi')$.  Then $L(s,\pi_1\times\widetilde{\pi}_1)$ and $L(s,\pi_2\times\widetilde{\pi}_2)$ are modular.  By Lemma \ref{lem:CG} again, if $\tau=\mathrm{Ad}(\pi)$, then $L(s,\pi_1\times\tau)/L(s,\pi_1)$ is entire.  If there exists $\pi_1\boxtimes\pi_2\in\mathfrak{F}_{9}$ such that such that $L(s,\pi_1\times\pi_2)=L(s,\pi_1\boxtimes\pi_2)$, then we can apply the Hoffstein--Ramakrishnan strategy to $\Pi = \mathbbm{1}\boxplus \mathrm{Ad}(\pi)\boxplus \pi_1\boxtimes\pi_2$ and $D(s)=L(s,\Pi\times\widetilde{\Pi})$.  However, no such representation $\pi_1\boxtimes\pi_2$ is known to exist yet.  If it did, then this particular choice of $\Pi$ would still not give us Theorem \ref{thm:Ad} when a nontrivial twist by $\chi$ is inserted because of complications arising from our inability to determine whether
\begin{equation}
\label{eqn:pole_question}
\textup{$\pi\not\sim\pi'$ implies that $\mathrm{Sym}^4(\pi)\otimes\overline{\omega}_{\pi}^2\neq \mathrm{Sym}^4(\pi')\otimes\overline{\omega}_{\pi'}^2$.}
\end{equation}

\subsection{The key auxiliary Dirichlet series}

Let $(\pi,\pi',\chi)\in\mathfrak{F}_2\times\mathfrak{F}_2\times\mathfrak{F}_1$.  Suppose that $\pi\not\sim\pi'$.  We introduce
\begin{equation}
\label{eqn:D_Maass}
\begin{aligned}
\mathcal{D}(s)&=\zeta_F(s)^6 L(s,\mathrm{Ad}(\pi)\times(\mathrm{Ad}(\pi')\otimes\chi))^4 L(s,\mathrm{Ad}(\pi)\times(\mathrm{Ad}(\pi')\otimes\overline{\chi}))^4\\
&\cdot L(s,\mathrm{Ad}(\pi))^7 L(s,\mathrm{Ad}(\pi'))^2 L(s,\mathrm{Ad}(\pi')\otimes\chi)^2 L(s,\mathrm{Ad}(\pi')\otimes\overline{\chi})^2\\
&\cdot L(s,\mathrm{Sym}^4(\pi)\otimes\overline{\omega}_{\pi}^{2})^5 L(s,\mathrm{Sym}^4(\pi')\otimes\overline{\omega}_{\pi'}^{2})^2 L(s,\mathrm{Ad}(\pi)\times\mathrm{Ad}(\pi'))^3\\
&\cdot L(s,\mathrm{Ad}(\pi)\times(\mathrm{Sym}^4(\pi')\otimes\overline{\omega}_{\pi'}^{2}))^3 L(s,\mathrm{Ad}(\pi')\times(\mathrm{Sym}^4(\pi)\otimes\overline{\omega}_{\pi}^{2}))\\
&\cdot L(s,\mathrm{Ad}(\pi')\times(\mathrm{Sym}^4(\pi)\otimes\chi \overline{\omega}_{\pi}^{2}))^2 L(s,\mathrm{Ad}(\pi')\times(\mathrm{Sym}^4(\pi)\otimes\overline{\chi}\,\overline{\omega}_{\pi}^{2}))^2 \\
&\cdot  L(s,\mathrm{Sym}^4(\pi)\times(\mathrm{Sym}^4(\pi')\otimes\overline{\omega}_{\pi}^{2} \overline{\omega}_{\pi'}^{2})).
\end{aligned}
\end{equation}
The Euler product defining $\mathcal{D}(s)$ for $\mathrm{Re}(s)>1$ has degree $324$, and it continues to an entire function apart from a pole at $s=1$.  The order of the pole at $s=1$ depends on $\pi$, $\pi'$, and $\chi$.  In particular, if $\pi\not\sim\pi'$, we do not know whether \eqref{eqn:pole_question} is true.  Therefore, we do not know whether $L(s,\mathrm{Sym}^4(\pi)\times(\mathrm{Sym}^4(\pi')\otimes\overline{\omega}_{\pi}^2\overline{\omega}_{\pi'}^2))$ has a pole.  Because of lack of progress towards the modularity of Rankin--Selberg $L$-functions, we have no analogue of \cite[Theorem 4.1.2]{Ramakrishnan2}.  Fortunately, $\mathcal{D}(s)$ is constructed so that this lack of progress will not hinder our proof.

Unlike \cite{Banks,HoffsteinLockhart,HoffsteinRamakrishnan,Luo,RamakrishnanWang}, the existence of an isobaric automorphic representation $\Pi$ such that $\mathcal{D}(s)=L(s,\Pi\times\widetilde{\Pi})$ is not yet known.  Therefore, we cannot use Lemma \ref{lem:HR} to establish the non-negativity of the Dirichlet coefficients $a_{\mathcal{D}}(v^{\ell})\log q_v$ of $-(\mathcal{D}'/D)(s)$ (recall \eqref{eqn:aDdef}).
\begin{lemma}
	\label{lem:coeff}
Let $(\pi,\pi',\chi)\in\mathfrak{F}_2\times\mathfrak{F}_2\times\mathfrak{F}_1$ with $\pi\not\sim\pi'$.  If $v\notin S_{\pi}\cup S_{\pi'}\cup S_{\chi}$, then $a_{\mathcal{D}}(v^{\ell})\geq 0$.
\end{lemma}
\begin{proof}
Fix $\ell\geq 1$ and a place $v\nmid\infty$ of $F$ such that $v\notin S_{\pi}\cup S_{\pi'}\cup S_{\chi}$ We use \eqref{eqn:a_def} throughout the proof without further mention.  Since $v$ and $\ell$ are fixed, we suppress $(v^{\ell})$ throughout (e.g., $a_{\mathcal{D}}=a_{\mathcal{D}}(v^{\ell})$ and $a_{\mathrm{Ad}(\pi)}=a_{\mathrm{Ad}(\pi)}(v^{\ell})$).  Since $\mathrm{Ad}(\pi)$ and $\mathrm{Ad}(\pi')$ are self-dual, it follows that $a_{\mathrm{Ad}(\pi)},a_{\mathrm{Ad}(\pi')}\in\mathbb{R}$.  A direct computation shows that $a_{\mathcal{D}}$ equals
\begin{align*}
&6+4a_{\mathrm{Ad}(\pi)\times(\mathrm{Ad}(\pi')\otimes\chi)}+4a_{\mathrm{Ad}(\pi)\times(\mathrm{Ad}(\pi')\otimes\overline{\chi})}+7a_{\mathrm{Ad}(\pi)}+2a_{\mathrm{Ad}(\pi')}\\
&+2a_{\mathrm{Ad}(\pi')\otimes\chi}+2a_{\mathrm{Ad}(\pi')\otimes\overline{\chi}}+5a_{\mathrm{Sym}^4(\pi)\otimes\overline{\omega}_{\pi}^{2}} + 2a_{\mathrm{Sym}^4(\pi')\otimes\overline{\omega}_{\pi'}^{2}} + 3a_{\mathrm{Ad}(\pi)\times\mathrm{Ad}(\pi')} \\
&+ 3a_{\mathrm{Ad}(\pi)\times(\mathrm{Sym}^4(\pi')\otimes\overline{\omega}_{\pi'}^{2})}+ a_{\mathrm{Ad}(\pi')\times(\mathrm{Sym}^4(\pi)\otimes\overline{\omega}_{\pi}^{2})}+ 2a_{\mathrm{Ad}(\pi')\times(\mathrm{Sym}^4(\pi)\otimes\chi \overline{\omega}_{\pi}^{2})}\\
&+2a_{\mathrm{Ad}(\pi')\times(\mathrm{Sym}^4(\pi)\otimes\overline{\chi}\,\overline{\omega}_{\pi}^{2})} + a_{\mathrm{Sym}^4(\pi)\times(\mathrm{Sym}^4(\pi')\otimes\overline{\omega}_{\pi}^{2} \overline{\omega}_{\pi'}^{2})}.
\end{align*}
Since $v\notin S_{\pi}\cup S_{\pi'}\cup S_{\chi}$, we can rewrite our expression for $a_{\mathcal{D}}$ as
\begin{align*}
&2 a_{\mathrm{Ad}(\pi')\otimes\chi} (a_{\mathrm{Sym}^4(\pi)\otimes\overline{\omega}_{\pi}^2}+a_{\mathrm{Ad}(\pi)}+1)+2 a_{\mathrm{Ad}(\pi')\otimes\overline{\chi}} (a_{\mathrm{Sym}^4(\pi)\otimes\overline{\omega}_{\pi}^2}+a_{\mathrm{Ad}(\pi)}+1)\\
&+2 a_{\mathrm{Ad}(\pi)} a_{\mathrm{Ad}(\pi')\otimes\chi}+2 a_{\mathrm{Ad}(\pi)} a_{\mathrm{Ad}(\pi')\otimes\overline{\chi}}+2 a_{\mathrm{Ad}(\pi)} (a_{\mathrm{Sym}^4(\pi')\otimes\overline{\omega}_{\pi'}^2}\\
&+a_{\mathrm{Ad}(\pi')}+1)+4 (a_{\mathrm{Sym}^4(\pi)\otimes\overline{\omega}_{\pi}^2}+a_{\mathrm{Ad}(\pi)}+1)+(a_{\mathrm{Sym}^4(\pi')\otimes\overline{\omega}_{\pi'}^2}+a_{\mathrm{Ad}(\pi')}+1)\\
&+(a_{\mathrm{Sym}^4(\pi)\otimes\overline{\omega}_{\pi}^2}+a_{\mathrm{Ad}(\pi)}+1)(a_{\mathrm{Sym}^4(\pi')\otimes\overline{\omega}_{\pi'}^2}+a_{\mathrm{Ad}(\pi')}+1)
\end{align*}
It follows from Lemma \ref{lem:CG} that if $v\notin S_{\pi}\cup S_{\pi'}\cup S_{\chi}$, then
\begin{align*}
a_{\mathrm{Ad}(\pi)}^2=1+a_{\mathrm{Ad}(\pi)}+a_{\mathrm{Sym}^4(\pi)\otimes\overline{\omega}_{\pi}^2},\qquad a_{\mathrm{Ad}(\pi')}^2=1+a_{\mathrm{Ad}(\pi')}+a_{\mathrm{Sym}^4(\pi')\otimes\overline{\omega}_{\pi'}^2}.
\end{align*}
Therefore, our expression for $a_{\mathcal{D}}$ simplifies to
\begin{align*}
&2a_{\mathrm{Ad}(\pi)}^2 a_{\mathrm{Ad}(\pi')\otimes\chi}+2a_{\mathrm{Ad}(\pi)}^2 a_{\mathrm{Ad}(\pi')\otimes\overline{\chi}}+2a_{\mathrm{Ad}(\pi)}a_{\mathrm{Ad}(\pi')\otimes\chi}+2a_{\mathrm{Ad}(\pi)}a_{\mathrm{Ad}(\pi')\otimes\overline{\chi}}\\
&\qquad +2 a_{\mathrm{Ad}(\pi)} a_{\mathrm{Ad}(\pi')}^2+4 a_{\mathrm{Ad}(\pi)}^2+a_{\mathrm{Ad}(\pi')}^2+a_{\mathrm{Ad}(\pi)}^2 a_{\mathrm{Ad}(\pi')}^2\\
&=4 \mathrm{Re}(a_{\text{Ad}(\pi)}^2 a_{\text{Ad}(\pi')\otimes\chi}+a_{\text{Ad}(\pi)} a_{\text{Ad}(\pi')\otimes\chi})\\
&\qquad+2 a_{\text{Ad}(\pi)} |a_{\text{Ad}(\pi')\otimes\chi}|^2+4 a_{\text{Ad}(\pi)}^2+|a_{\text{Ad}(\pi')\otimes\chi}|^2+a_{\text{Ad}(\pi)}^2 |a_{\text{Ad}(\pi')\otimes\chi}|^2\\
&=|2a_{\mathrm{Ad}(\pi)}+a_{\mathrm{Ad}(\pi)}a_{\mathrm{Ad}(\pi')\otimes\chi}+a_{\mathrm{Ad}(\pi')\otimes\chi}|^2,
\end{align*}
which is nonnegative.
\end{proof}
\section{Proof of Theorem \ref{thm}:  non-dihedral and twist-inequivalent cases}
\label{sec:thm_part1}

Define $\mathfrak{F}_{2}^{\flat}=\{\pi\in\mathfrak{F}_{2}\colon \textup{$\pi$ non-dihedral}\}$.  We require the following result.

\begin{lemma}\cite{GJ,Kim,KimShahidi2,KimShahidi,Ramakrishnan}
\label{lem:decomp}
\begin{enumerate}[(i)]
	\item If $\pi\in\mathfrak{F}_{2}$, then $\mathrm{Ad}(\pi)\in\mathfrak{A}_{3}$, $\mathrm{Sym}^3(\pi)\in\mathfrak{A}_{4}$, and $\mathrm{Sym}^4(\pi)\in\mathfrak{A}_{5}$.
	\item If $\pi\in\mathfrak{F}_{2}$, then $\mathrm{Ad}(\pi)\notin\mathfrak{F}_{3}$ if and only if $\pi\notin\mathfrak{F}_{2}^{\flat}$, in which case there exists an idele class character $\xi_{\pi}$ of a quadratic extension $K_{\pi}/F$ such that $\pi=\mathrm{Ind}_{K_{\pi}}^{F}(\xi_{\pi})$ and
	\[
	\mathrm{Ad}(\pi) = \mathrm{Ind}_{K_{\pi}}^{F}(\xi_{\pi}^2)\otimes \omega_{\pi}\boxplus \xi_{\pi}|_F \omega_{\pi}.
	\]
	\item Let $\pi\in\mathfrak{F}_{2}^{\flat}$, in which case $\mathrm{Ad}(\pi)\in\mathfrak{F}_{3}$.
	\begin{enumerate}
	\item $\mathrm{Sym}^3(\pi)\notin\mathfrak{F}_{4}$ if and only if there exists $\mu_{\pi}\in\mathfrak{F}_{1}-\{\mathbbm{1}\}$ such that
	\[
	\mu_{\pi}^3=\mathbbm{1},\qquad \mathrm{Ad}(\pi)=\mathrm{Ad}(\pi)\otimes\mu_{\pi},\qquad \mathrm{Sym}^4(\pi)\otimes\overline{\omega}_{\pi}^{2} = \mathrm{Ad}(\pi)\boxplus \mu_{\pi}\boxplus\overline{\mu}_{\pi}.
	\]
	\item If $\mathrm{Sym}^3(\pi)\in\mathfrak{F}_{4}$, then $\mathrm{Sym}^4(\pi)\notin\mathfrak{F}_{5}$ if and only if there exists $\eta_{\pi}\in\mathfrak{F}_{1}-\{\mathbbm{1}\}$ and a dihedral $\nu_{\pi}\in\mathfrak{F}_{2}$ such that
	\[
	\eta_{\pi}^2=\mathbbm{1},\qquad \mathrm{Sym}^3(\pi)=\mathrm{Sym}^3(\pi)\otimes\eta_{\pi},\qquad \mathrm{Sym}^4(\pi)\otimes\overline{\omega}_{\pi}^{2} =  \nu_{\pi} \boxplus \mathrm{Ad}(\pi)\otimes\eta_{\pi}.
	\]
	\end{enumerate}
\end{enumerate}
\end{lemma}

In this section, we prove the hardest cases of Theorem \ref{thm}, when $\pi,\pi'\in\mathfrak{F}_{2}^{\flat}$ and $\pi\not\sim\pi'$.  Consider $\mathcal{D}(s)$ in \eqref{eqn:D_Maass}.  Per Lemma \ref{lem:coeff}, at unramified places $v\nmid\infty$, the Dirichlet coefficients of $-(\mathcal{D}'/\mathcal{D})(s)$ are non-negative.  We must determine which factors other $\zeta_F(s)^6$ have poles.  Per Theorem \ref{thm:JPSS}, only $L(s,\mathrm{Sym}^4(\pi)\times(\mathrm{Sym}^4(\pi')\otimes\overline{\omega}_{\pi}^2\overline{\omega}_{\pi'}^2))$ and the twists of $L(s,\mathrm{Ad}(\pi)\times\mathrm{Ad}(\pi'))$ might have poles.
\begin{lemma}
\label{lem:Ramakrishnan_multiplicity_one}
Let $\pi,\pi'\in\mathfrak{F}_{2}^{\flat}$ and $\xi=\xi^*|\cdot|^{it_{\xi}}\in\mathfrak{F}_{1}$.  If
\begin{itemize}
	\item ${\xi^*}^3\neq\mathbbm{1}$, or
	\item ${\xi^*}^3=\mathbbm{1}$, $\xi^*\neq\mathbbm{1}$, $\mathrm{Ad}(\pi)\neq\mathrm{Ad}(\pi)\otimes\xi^*$, and $\mathrm{Ad}(\pi')\neq\mathrm{Ad}(\pi')\otimes\xi^*$, or
	\item $\xi^*=\mathbbm{1}$ and $\pi\not\sim\pi'$,
\end{itemize}
then $L(s,\mathrm{Ad}(\pi)\times(\mathrm{Ad}(\pi')\otimes\xi))$ is entire.
\end{lemma}
\begin{proof}
We prove the contrapositive.  If $L(s,\mathrm{Ad}(\pi)\times(\mathrm{Ad}(\pi')\otimes\xi))$ has a pole, then the self-duality of $\mathrm{Ad}(\pi)$ and $\mathrm{Ad}(\pi')$ imply that $\mathrm{Ad}(\pi)=\mathrm{Ad}(\pi')\otimes\xi^*$.  This equality implies that $\mathbbm{1}=\omega_{\mathrm{Ad}(\pi)}=\omega_{\mathrm{Ad}(\pi')\otimes\xi^*}={\xi^*}^3$.  Since ${\xi^*}^3=\mathbbm{1}$, $\xi^*$ is not self-dual unless $\xi^*=\mathbbm{1}$.  Therefore, if $\xi^*\neq\mathbbm{1}$, then $\mathrm{Ad}(\pi)=\mathrm{Ad}(\pi)\otimes\xi^*$ or $\mathrm{Ad}(\pi')=\mathrm{Ad}(\pi')\otimes\xi^*$.  If $\xi^*=\mathbbm{1}$, then $\mathrm{Ad}(\pi)=\mathrm{Ad}(\pi')$, in which case $\pi\sim\pi'$ per \cite[Theorem 4.1.2]{Ramakrishnan}.
\end{proof}

We proceed to find integers $\ell\geq 1$ and $k\geq 1$ such that $2\ell>k$ and
\begin{equation}
\label{eqn:ell_and_k}
(s-1)^k\mathcal{D}(s)L(s,\mathrm{Ad}(\pi)\times(\mathrm{Ad}(\pi')\otimes\chi))^{-\ell} L(s,\mathrm{Ad}(\pi)\times(\mathrm{Ad}(\pi')\otimes\overline{\chi}))^{-\ell}
\end{equation}
is entire.  Theorem \ref{thm:Ad} then follows from Lemma \ref{lem:GHL}.  To find such integers $\ell$ and $k$, we proceed by casework, as dictated by Lemma \ref{lem:decomp}.

\subsection{$\mathrm{Sym}^3(\pi),\mathrm{Sym}^3(\pi')\notin\mathfrak{F}_{4}$}

We apply Proposition \ref{prop:list}(iiia) to $\mathrm{Sym}^4(\pi)\otimes\overline{\omega}_{\pi}^{2}$ and $\mathrm{Sym}^4(\pi')\otimes\overline{\omega}_{\pi'}^{2}$.  It follows that if $k=10$, $\ell=6$, and
\begin{align*}
D_1(s)&= \zeta_F(s)^6 L(s,\mu_{\pi}\mu_{\pi'}) L(s,\overline{\mu}_{\pi}\overline{\mu}_{\pi'})L(s,\mu_{\pi}\overline{\mu}_{\pi'}) L(s,\overline{\mu}_{\pi}\mu_{\pi'}),\\
D_2(s)&= L(s,\mathrm{Ad}(\pi))^{12} L(s,\mathrm{Ad}(\pi'))^4 L(s,\mathrm{Ad}(\pi)\otimes\mu_{\pi'})^4 L(s,\mathrm{Ad}(\pi)\otimes\overline{\mu}_{\pi'})^4 L(s,\mu_{\pi})^5 L(s,\overline{\mu}_{\pi})^5\\
&\cdot  L(s,\mathrm{Ad}(\pi')\otimes\mu_{\pi})^2 L(s,\mathrm{Ad}(\pi')\otimes\overline{\mu}_{\pi})^2 L(s,\mathrm{Ad}(\pi')\otimes\chi)^2  L(s,\mathrm{Ad}(\pi')\otimes\overline{\chi})^2 L(s,\mu_{\pi'})^2 \\
&\cdot L(s,\overline{\mu}_{\pi'})^2 L(s,\mathrm{Ad}(\pi')\otimes\chi\mu_{\pi})^2 L(s,\mathrm{Ad}(\pi')\otimes\overline{\chi}\mu_{\pi})^2 L(s,\mathrm{Ad}(\pi')\otimes\chi\overline{\mu}_{\pi})^2 \\
&\cdot L(s,\mathrm{Ad}(\pi')\otimes\overline{\chi}\,\overline{\mu}_{\pi})^2 L(s,\mathrm{Ad}(\pi)\times\mathrm{Ad}(\pi'))^8, 
\end{align*}
then \eqref{eqn:ell_and_k} equals $D_1(s) D_2(s) (s-1)^{10}$.  Since $\mu_{\pi}^3=\mu_{\pi'}^3=\mathbbm{1}$, this is entire by Lemma \ref{lem:Ramakrishnan_multiplicity_one}.

\subsection{$\mathrm{Sym}^3(\pi)\notin\mathfrak{F}_{4}$, $\mathrm{Sym}^3(\pi')\in\mathfrak{F}_{4}$, and $\mathrm{Sym}^4(\pi')\notin\mathfrak{F}_{5}$}


We apply Proposition \ref{prop:list}(iiia,iiib) to $\mathrm{Sym}^4(\pi)\otimes\overline{\omega}_{\pi}^{2}$  and $\mathrm{Sym}^4(\pi')\otimes\overline{\omega}_{\pi'}^{2}$. We conclude that if $k=6$, $\ell=6$, and
\begin{align*}
D_3(s)&=L(s,\mathrm{Ad}(\pi))^{12} L(s,\mathrm{Ad}(\pi'))^{2} L(s, \nu_{\pi'})^2 L(s,\mathrm{Ad}(\pi)\times \nu_{\pi'})^4 L(s,\mathrm{Ad}(\pi')\otimes\eta_{\pi'})^2\\
&\cdot L(s,\mathrm{Ad}(\pi)\times(\mathrm{Ad}(\pi')\otimes\eta_{\pi'}))^4 L(s,\mu_{\pi})^5 L(s,\overline{\mu}_{\pi})^5 L(s,\nu_{\pi'}\otimes\mu_{\pi}) L(s,\nu_{\pi'}\otimes\overline{\mu}_{\pi})\\
&\cdot  L(s,\mathrm{Ad}(\pi')\otimes\mu_{\pi}) L(s,\mathrm{Ad}(\pi')\otimes\overline{\mu}_{\pi}) L(s,\mathrm{Ad}(\pi')\otimes\mu_{\pi}\eta_{\pi'}) L(s,\mathrm{Ad}(\pi')\otimes\overline{\mu}_{\pi}\eta_{\pi'})\\
&\cdot L(s,\mathrm{Ad}(\pi)\times\mathrm{Ad}(\pi'))^4 L(s,\mathrm{Ad}(\pi')\otimes\chi)^2 L(s,\mathrm{Ad}(\pi')\otimes\overline{\chi})^2 L(s,\mathrm{Ad}(\pi')\otimes\chi\mu_{\pi})^2\\
&\cdot  L(s,\mathrm{Ad}(\pi')\otimes\overline{\chi}\mu_{\pi})^2 L(s,\mathrm{Ad}(\pi')\otimes\chi\overline{\mu}_{\pi})^2 L(s,\mathrm{Ad}(\pi')\otimes\overline{\chi}\,\overline{\mu}_{\pi})^2,
\end{align*}
then \eqref{eqn:ell_and_k} equals $\zeta_F(s)^6 (s-1)^6 D_3(s)$.  Since $\mu_{\pi}^3=\eta_{\pi'}^2=\mathbbm{1}$, this is entire by Lemma \ref{lem:Ramakrishnan_multiplicity_one}.

\subsection{$\mathrm{Sym}^3(\pi),\mathrm{Sym}^3(\pi')\in\mathfrak{F}_{4}$ and $\mathrm{Sym}^4(\pi),\mathrm{Sym}^4(\pi')\notin\mathfrak{F}_{5}$}

We apply Proposition \ref{prop:list}(iiib) to $\mathrm{Sym}^4(\pi)\otimes\overline{\omega}_{\pi}^2$ and $\mathrm{Sym}^4(\pi')\otimes\overline{\omega}_{\pi'}^2$.  If $k=7$, $\ell=4$, and
\begin{align*}
D_4(s)&=L(s,\mathrm{Ad}(\pi))^7 L(s,\mathrm{Ad}(\pi'))^2  L(s, \nu_{\pi}\times\mathrm{Ad}(\pi')) L(s,\mathrm{Ad}(\pi)\otimes\eta_{\pi})^5 L(s,\mathrm{Ad}(\pi)\times \nu_{\pi'})^3\\
&\cdot L(s,\nu_{\pi'})^2 L(s,\nu_{\pi})^5 L(s,\mathrm{Ad}(\pi)\times(\nu_{\pi'}\otimes\eta_{\pi}))L(s,\mathrm{Ad}(\pi')\otimes\eta_{\pi'})^2 L(s,\mathrm{Ad}(\pi')\times( \nu_{\pi}\otimes\eta_{\pi'}))\\
&\cdot  L(s,\mathrm{Ad}(\pi)\times(\mathrm{Ad}(\pi')\otimes\eta_{\pi})) L(s,\mathrm{Ad}(\pi')\otimes\chi)^2 L(s,\mathrm{Ad}(\pi')\otimes\overline{\chi})^2 L(s,\mathrm{Ad}(\pi')\times(\nu_{\pi}\otimes\chi))^2\\
&\cdot  L(s,\mathrm{Ad}(\pi)\times(\mathrm{Ad}(\pi')\otimes\eta_{\pi'}))^3  L(s,\mathrm{Ad}(\pi')\times(\nu_{\pi}\otimes\overline{\chi}))^2 L(s,\mathrm{Ad}(\pi)\times(\mathrm{Ad}(\pi')\otimes\eta_{\pi}\eta_{\pi'}))  \\
&\cdot L(s,\mathrm{Ad}(\pi)\times\mathrm{Ad}(\pi'))^3 L(s,\mathrm{Ad}(\pi)\times(\mathrm{Ad}(\pi')\otimes\chi\eta_{\pi}))^2 L(s,\mathrm{Ad}(\pi)\times(\mathrm{Ad}(\pi')\otimes\overline{\chi}\eta_{\pi}))^2,
\end{align*}
then \eqref{eqn:ell_and_k} equals $\zeta_F(s)^6 L(s,\nu_{\pi}\times\nu_{\pi'}) (s-1)^7 D_4(s)$.  Since $\nu_{\pi}$ and $\nu_{\pi'}$ are cuspidal, $L(s,\nu_{\pi}\times\nu_{\pi'})$ has a pole of order at most $1$ at $s=1$.  Therefore, by Lemma \ref{lem:Ramakrishnan_multiplicity_one}, $\zeta_F(s)^6 L(s,\nu_{\pi}\times\nu_{\pi'}) (s-1)^7 D_4(s)$ is entire.

\subsection{$\mathrm{Sym}^3(\pi)\in\mathfrak{F}_{4}$, $\mathrm{Sym}^4(\pi)\in\mathfrak{F}_{5}$, and $\pi'\in\mathfrak{F}_{2}^{\flat}$}


\subsubsection{$\mathrm{Sym}^3(\pi')\notin\mathfrak{F}_{4}$}
If $k=6$, $\ell=4$, and $D_5(s)$ equals
\begin{align*}
&L(s,\mathrm{Ad}(\pi))^7 L(s,\mathrm{Ad}(\pi'))^4 L(s,\mu_{\pi'})^2 L(s,\overline{\mu}_{\pi'})^2 L(s,\mathrm{Ad}(\pi)\otimes\mu_{\pi'})^3 L(s,\mathrm{Ad}(\pi)\otimes\overline{\mu}_{\pi'})^3\\
&\cdot L(s,\mathrm{Sym}^4(\pi)\otimes\overline{\omega}_{\pi}^{2})^5 L(s,\mathrm{Sym}^4(\pi)\otimes\overline{\omega}_{\pi}^{2}\mu_{\pi'}) L(s,\mathrm{Sym}^4(\pi)\otimes\overline{\omega}_{\pi}^{2} \overline{\mu}_{\pi'})L(s,\mathrm{Ad}(\pi')\otimes\chi)^2 \\
&\cdot L(s,\mathrm{Sym}^4(\pi)\times(\mathrm{Ad}(\pi')\otimes\chi \overline{\omega}_{\pi}^{2}))^2 L(s,\mathrm{Sym}^4(\pi)\times(\mathrm{Ad}(\pi')\otimes\overline{\chi} \overline{\omega}_{\pi}^{2}))^2 L(s,\mathrm{Ad}(\pi')\otimes\overline{\chi})^2\\
&\cdot L(s,\mathrm{Sym}^4(\pi)\times(\mathrm{Ad}(\pi')\otimes\overline{\omega}_{\pi}^{2}))^2 L(s,\mathrm{Ad}(\pi)\times\mathrm{Ad}(\pi'))^6,
\end{align*}
then by Proposition \ref{prop:list}(iiia) applied to $\mathrm{Sym}^4(\pi')\otimes\overline{\omega}_{\pi'}^{2}$,  \eqref{eqn:ell_and_k} equals $\zeta_F(s)^6 (s-1)^6 D_5(s)$.  Since $\mu_{\pi'}^3=\mathbbm{1}$ and $\mu_{\pi'}\neq\mathbbm{1}$, it follows from Lemma \ref{lem:Ramakrishnan_multiplicity_one} that $\zeta_F(s)^6 (s-1)^6 D_5(s)$ is entire.

\subsubsection{$\mathrm{Sym}^3(\pi')\in\mathfrak{F}_{4}$ and $\mathrm{Sym}^4(\pi')\notin\mathfrak{F}_{5}$}
We apply Proposition \ref{prop:list}(iiib) to $\mathrm{Sym}^4(\pi')\otimes\overline{\omega}_{\pi'}^{2}$ and conclude that if $k=6$, $\ell=4$, and $D_6(s)$ equals
\begin{align*}
&L(s,\mathrm{Sym}^4(\pi)\otimes\overline{\omega}_{\pi}^{2})^5 L(s,\mathrm{Ad}(\pi')\otimes\eta_{\pi'})^2  L(s,\mathrm{Ad}(\pi)\times \nu_{\pi'})^3 L(s,\mathrm{Ad}(\pi'))^2 L(s,\nu_{\pi'})^2 L(s,\mathrm{Ad}(\pi))^7\\
&\cdot L(s,\mathrm{Ad}(\pi)\times(\mathrm{Ad}(\pi')\otimes \eta_{\pi'}))^3   L(s,\mathrm{Sym}^4(\pi)\times(\mathrm{Ad}(\pi')\otimes\overline{\omega}_{\pi}^{2} \eta_{\pi'})) L(s,\mathrm{Sym}^4(\pi)\times(\nu_{\pi'}\otimes\overline{\omega}_{\pi}^{2}))\\
&\cdot L(s,\mathrm{Sym}^4(\pi)\times(\mathrm{Ad}(\pi')\otimes\overline{\omega}_{\pi}^{2})) L(s,\mathrm{Ad}(\pi)\times\mathrm{Ad}(\pi'))^3 L(s,\mathrm{Ad}(\pi')\otimes\chi)^2\\
&\cdot  L(s,\mathrm{Sym}^4(\pi)\times(\mathrm{Ad}(\pi')\otimes\chi\overline{\omega}_{\pi}^{2}))^2 L(s,\mathrm{Ad}(\pi')\otimes\overline{\chi})^2 L(s,\mathrm{Sym}^4(\pi)\times(\mathrm{Ad}(\pi')\otimes\overline{\chi}\,\overline{\omega}_{\pi}^{2}))^2,
\end{align*}
then \eqref{eqn:ell_and_k} equals $\zeta_F(s)^6 (s-1)^6  D_6(s)$.  By Lemma \ref{lem:Ramakrishnan_multiplicity_one}, this is entire.

\subsubsection{$\mathrm{Sym}^3(\pi')\in\mathfrak{F}_{4}$ and $\mathrm{Sym}^4(\pi')\in\mathfrak{F}_{5}$}
If $k=7$, $\ell=4$, and
\begin{align*}
D_7(s)&=L(s,\mathrm{Ad}(\pi')\otimes\chi)^2 L(s,\mathrm{Ad}(\pi')\otimes\overline{\chi})^2 L(s,\mathrm{Sym}^4(\pi)\otimes\overline{\omega}_{\pi}^{2})^5 L(s,\mathrm{Ad}(\pi))^7 L(s,\mathrm{Ad}(\pi'))^2\\
&\cdot L(s,\mathrm{Ad}(\pi)\times\mathrm{Ad}(\pi'))^3 L(s,\mathrm{Ad}(\pi)\times(\mathrm{Sym}^4(\pi')\otimes\overline{\omega}_{\pi'}^{2}))^3 L(s,\mathrm{Sym}^4(\pi')\otimes\overline{\omega}_{\pi'}^{2})^2\\
&\cdot L(s,\mathrm{Ad}(\pi')\times(\mathrm{Sym}^4(\pi)\otimes\overline{\omega}_{\pi}^{2})) L(s,\mathrm{Ad}(\pi')\times(\mathrm{Sym}^4(\pi)\otimes\overline{\omega}_{\pi}^{2}))^4,
\end{align*}
then \eqref{eqn:ell_and_k} equals $\zeta_F(s)^6 (s-1)^7 L(s,\mathrm{Sym}^4(\pi)\times(\mathrm{Sym}^4(\pi')\otimes\overline{\omega}_{\pi}^{2}\overline{\omega}_{\pi'}^{2})) D_7(s)$.  Because both $\mathrm{Sym}^4(\pi)$ and $\mathrm{Sym}^4(\pi')$ are cuspidal, the $L$-function $L(s,\mathrm{Sym}^4(\pi)\times(\mathrm{Sym}^4(\pi')\otimes\overline{\omega}_{\pi}^{2}\overline{\omega}_{\pi'}^{2}))$ has a pole of order $0$ or $1$ at $s=1$.  Therefore, by Lemma \ref{lem:Ramakrishnan_multiplicity_one}, $D_7(s)$ is entire.

\section{Proof of Theorem \ref{thm}: The remaining cases}
\label{sec:thm_part2}

For the remaining cases, $L(s,\mathrm{Ad}(\pi)\times(\mathrm{Ad}(\pi')\otimes\chi))$ factors via Lemmata \ref{lem:CG} and \ref{lem:decomp}.  We exhaustively examine these remaining cases.

\subsection{$\pi,\pi'\in\mathfrak{F}_{2}-\mathfrak{F}_{2}^{\flat}$}
By Lemma \ref{lem:decomp}(ii), $L(s,\mathrm{Ad}(\pi)\times(\mathrm{Ad}(\pi')\otimes\chi))$ is a product of $\mathrm{GL}_m\times\mathrm{GL}_n$ $L$-functions with $m,n \leq 2$. By Proposition \ref{prop:list}(i,ii,iii), Theorem \ref{thm} follows.

\subsection{$\pi\in\mathfrak{F}_{2}^{\flat}$ and $\pi'\in \mathfrak{F}_{2}-\mathfrak{F}_{2}^{\flat}$.}

By Lemma \ref{lem:decomp}(ii), $\mathrm{Ad}(\pi)$ is cuspidal and
\[
L(s,\mathrm{Ad}(\pi)\times(\mathrm{Ad}(\pi')\otimes\chi))=L(s,\mathrm{Ind}_{K_{\pi'}}^F(\xi_{\pi'}^2)\times(\mathrm{Ad}(\pi)\otimes\chi\overline{\omega}_{\pi'})) L(s,\mathrm{Ad}(\pi)\otimes\chi\overline{\omega}_{\pi'}\xi_{\pi'}|_F).
\]
Note that $\mathrm{Ad}(\pi)\not\sim\mathrm{Ad}(\mathrm{Ind}_{K_{\pi'}}^{F}(\xi_{\pi'}^2))$.  Now, Theorem \ref{thm} follows from Proposition \ref{prop:list}(ii,v).

\subsection{$\pi,\pi'\in\mathfrak{F}_{2}^{\flat}$ and $\pi\sim\pi'$}

We have that $\mathrm{Ad}(\pi)=\mathrm{Ad}(\pi')$ and $\mathrm{Ad}(\pi)\in\mathfrak{F}_{3}$.  By Lemma \ref{lem:CG}, $L(s,\mathrm{Ad}(\pi)\times(\mathrm{Ad}(\pi')\otimes\chi))$ factors as $L(s,\chi) L(s,\mathrm{Ad}(\pi)\otimes\chi)L(s,\mathrm{Sym}^4(\pi)\otimes\chi\overline{\omega}_{\pi}^{2})$ unless $\mathrm{Sym}^3(\pi)\notin\mathfrak{F}_{4}$ and $\chi=\chi^*\otimes|\cdot|^{it_{\chi}}$ satisfies $\mathrm{Ad}(\pi)\otimes\chi^*=\mathrm{Ad}(\pi)$, in which case ${\chi^*}^3=\mathbbm{1}$, $\chi^*\neq\mathbbm{1}$, and
\[
L(s,\mathrm{Ad}(\pi)\times(\mathrm{Ad}(\pi')\otimes\chi)) = \zeta_F(s+it_{\chi})L(s,\mathrm{Ad}(\pi)\otimes|\cdot|^{it_{\chi}})L(s,\mathrm{Sym}^4(\pi)\otimes\overline{\omega}_{\pi}^{2}|\cdot|^{it_{\chi}}).
\]
In light of Proposition \ref{prop:list}(ii), Theorem \ref{thm} follows from a determination of whether the pertinent twist of $L(s,\mathrm{Sym}^4(\pi))$ has a exceptional zero.

\subsubsection{$\mathrm{Sym}^3(\pi)\notin\mathfrak{F}_{4}$}
Let $\chi=\chi^*|\cdot|^{it_{\chi}}\in\mathfrak{F}_{1}$.  If $\chi^*\in\{\mu_{\pi},\overline{\mu}_{\pi}\}$, then
\[
L(s,\mathrm{Sym}^4(\pi)\otimes\overline{\omega}_{\pi}^2|\cdot|^{it_{\chi}}) = L(s,\mathrm{Ad}(\pi)\otimes|\cdot|^{it_{\chi}}) L(s,\mu_{\pi}|\cdot|^{it_{\chi}})L(s,\overline{\mu}_{\pi}|\cdot|^{it_{\chi}}).
\]
by Proposition \ref{prop:list}(iiia).  Since $\mu_{\pi}$ is cubic, $L(s,\mathrm{Sym}^4(\pi)\otimes\overline{\omega}_{\pi}^2|\cdot|^{it_{\chi}})$ has no exceptional zero by Proposition \ref{prop:list}(i,ii).  If $\chi^*\notin\{\mu_{\pi},\overline{\mu}_{\pi}\}$, then
\[
L(s,\mathrm{Sym}^4(\pi)\otimes\chi\overline{\omega}_{\pi}^2)=L(s,\mathrm{Ad}(\pi)\otimes\chi) L(s,\chi\overline{\mu}_{\pi})L(s,\chi\mu_{\pi}).
\]
By Proposition \ref{prop:list}(i,ii), an exceptional zero of $L(s,\mathrm{Sym}^4(\pi)\otimes\chi\overline{\omega}_{\pi}^2)$ is necessarily an exceptional zero of $L(s,\chi\overline{\mu}_{\pi})L(s,\chi\mu_{\pi})$.

\subsubsection{$\mathrm{Sym}^3(\pi)\in\mathfrak{F}_{4}$ and $\mathrm{Sym}^4(\pi)\notin\mathfrak{F}_{5}$}

Per Lemma \ref{lem:decomp}(iiib), the factorization
\[
L(s,\mathrm{Sym}^4(\pi)\otimes\chi\overline{\omega}_{\pi}^{2})=L(s,\nu_{\pi}\otimes\chi)L(s,\mathrm{Ad}(\pi)\otimes\eta_{\pi}\chi)
\]
holds.  Since $\nu_{\pi}$ is cuspidal, no exceptional zero exists per Proposition \ref{prop:list}(ii).

\subsubsection{$\mathrm{Sym}^3(\pi)\in\mathfrak{F}_{4}$ and $\mathrm{Sym}^4(\pi)\in\mathfrak{F}_{5}$}

Define $\Pi_1 = \mathbbm{1}\boxplus\mathrm{Ad}(\pi)\boxplus\mathrm{Sym}^4(\pi)\otimes\chi\overline{\omega}_{\pi}^{2}$.  Both $\mathrm{Ad}(\pi)$ and $\mathrm{Sym}^4(\pi)\otimes\overline{\omega}_{\pi}^2$ are self-dual.  The $L$-function $D_8(s)=L(s,\Pi_1\times\widetilde{\Pi}_1)$ factors as
\begin{align*}
D_8(s) &=\zeta_F(s)L(s,\mathrm{Ad}(\pi)\times\mathrm{Ad}(\pi)) L(s,\mathrm{Sym}^4(\pi)\times\mathrm{Sym}^4(\widetilde{\pi}))\\
&\cdot L(s,\mathrm{Ad}(\pi))^2L(s,\mathrm{Ad}(\pi)\times(\mathrm{Sym}^4(\pi)\otimes\chi\overline{\omega}_{\pi}^{2})) L(s,\mathrm{Ad}(\pi)\times(\mathrm{Sym}^4(\pi)\otimes\overline{\chi}\,\overline{\omega}_{\pi}^{2}))\\
&\cdot L(s,\mathrm{Sym}^4(\pi)\otimes\chi\overline{\omega}_{\pi}^{2}) L(s,\mathrm{Sym}^4(\pi)\otimes\overline{\chi}\,\overline{\omega}_{\pi}^{2})
\end{align*}
by Lemma \ref{lem:CG}.  Using Lemma \ref{lem:HR}, we apply Lemma \ref{lem:GHL} to $D_8(s)$ with $\log Q\ll \log (C(\pi)C(\chi))$ and $r=3$.  It follows that $D_8(s)$ has at most $3$ zeros in $I = [1-\Cr{GHL}/\log Q,1)$.

We claim that the functions
\[
\mathcal{H}(s)=\frac{L(s,\mathrm{Ad}(\pi)\times(\mathrm{Sym}^4(\pi)\otimes\chi\overline{\omega}_{\pi}^{2}))}{L(s,\mathrm{Sym}^4(\pi)\otimes\chi\overline{\omega}_{\pi}^{2})},\quad \overline{\mathcal{H}(\overline{s})}=\frac{L(s,\mathrm{Ad}(\pi)\times(\mathrm{Sym}^4(\pi)\otimes\overline{\chi}\,\overline{\omega}_{\pi}^{2}))}{L(s,\mathrm{Sym}^4(\pi)\otimes\overline{\chi}\,\overline{\omega}_{\pi}^{2})}
\]
are holomorphic when $\mathrm{Re}(s)\geq \frac{56}{65}$.  Thus, by \eqref{eqn:dual}, any real zero of $L(\sigma,\mathrm{Sym}^4(\pi)\otimes\chi\overline{\omega}_{\pi}^{2})$ in $I$ is a real zero of $D_8(\sigma)$ in $I$ of order at least 4, a contradiction.  Therefore, no such zero exists.  It suffices for us to prove the claim for $\mathcal{H}(s)$.

To prove the claim, let $S=S_{\pi}^{\infty}\cup S_{\chi}^{\infty}\cup\{v\mid 2\}$, which contains $S_{\mathrm{Sym}^3(\pi)}^{\infty}\cup S_{\chi}^{\infty}\cup\{v\mid 2\}$ by Lemma \ref{lem:CG}.  By \eqref{eqn:symm_power_def} and \eqref{eqn:SYM2def}, if $\mathrm{Re}(s)>1$, then
\begin{equation}
\label{eqn:H_def}
\mathcal{H}(s)=L^{S}(s,\mathrm{Sym}^3(\pi);\mathrm{Sym}^2\otimes\chi\overline{\omega}_{\pi}^{3})\prod_{v\in S,~v\nmid\infty} \frac{L(s,\mathrm{Ad}(\pi_v) \times\mathrm{Sym}^4(\pi_v)\otimes\chi_v\,\overline{\omega}_{\pi,v}^{2})}{L(s,\mathrm{Sym}^4(\pi_v)\otimes\chi_v\,\overline{\omega}_{\pi,v}^{2})}.
\end{equation}
By \eqref{eqn:Ramanujan1} and \eqref{eqn:Ramanujan2}, the product over $v\in S$ in \eqref{eqn:H_def} is holomorphic when $\mathrm{Re}(s)\geq \frac{56}{65}$.  The numerator and denominator of $\mathcal{H}$ are holomorphic and non-vanishing at $s=1$.  Therefore, by Theorem \ref{thm:sym2}, $L^S(s,\mathrm{Sym}^3(\pi);\mathrm{Sym}^2\otimes\chi\overline{\omega}_{\pi}^3)$ is holomorphic for $\mathrm{Re}(s)>0$.  The claim now follows from the identity theorem in complex analysis.

\bibliographystyle{abbrv}
\bibliography{JAThorner_SiegelZeros}

\end{document}